\documentclass[10pt, article]{amsart}
\usepackage{ae} 
\usepackage[T1]{fontenc}
\usepackage[cp1250]{inputenc}
\usepackage{amsmath}
\usepackage{amssymb, amsfonts,amscd,verbatim}
\usepackage{mathtools}

\usepackage{tikz-cd}

\usepackage[normalem]{ulem}
\usepackage{hyperref}
\usepackage{indentfirst}
\usepackage{latexsym}
\input xy
\xyoption{all}

\usepackage{amsmath}    

\theoremstyle{plain}
\newtheorem{Pocz}{Poczatek}[section]
\newtheorem{Proposition}[Pocz]{Proposition}

\newtheorem{Theorem}[Pocz]{Theorem}
\newtheorem{Corollary}[Pocz]{Corollary}

\newtheorem{Lemma}[Pocz]{Lemma}

\newtheorem{Question}[Pocz]{Question}

\newtheorem{Example}[Pocz]{Example}

\theoremstyle{definition}
\newtheorem{Definition}[Pocz]{Definition}

\theoremstyle{remark}
\newtheorem{Remark}[Pocz]{Remark}

\DeclareMathOperator*{\diam}{diam}

\def\RR{{\mathbb R}}
\def\CC{{\mathbb C}}

\def\ZZ{{\mathbb Z}}
\def\NN{{\mathbb N}}

\def\f{{\varphi}}
\def\r{{\varrho}}

\def\eps{\varepsilon}
\def\U{\mathcal U}
\def\V{\mathcal V}
\def\A{\mathcal A}
\def\B{\mathcal B}
\def\f{\varphi}
\def\e{\varepsilon}

\def\asdim{\mathrm{asdim}}

\def\dim{\mathrm{dim}}
\def\diam{\mathrm{diam}}

\numberwithin{equation}{section}

\title[Higson Compactification and Dimension Raising 
]%
  {Higson Compactification and Dimension Raising}

\author{Kyle ~Austin}
\address{Ben Gurion University of the Negev, Beer Sheva, Israel}
\email{kyle@math.bgu.ac.il}

\author{\v Ziga ~Virk}
\address{IST Austria}
\email{ziga.virk@ist.ac.at}
\thanks{The first named author was supported by the ISRAEL SCIENCE FOUNDATION (grant No. 522/14).}

\date{ \today
}


\begin{document}
\maketitle

\begin{abstract}
Let $X$ and $Y$ be proper metric spaces. We show that a coarsely $n$-to-$1$ map $f\colon X\to Y$ induces an $n$-to-$1$ map of Higson coronas. This viewpoint turns out to be successful in showing that the classical dimension raising theorems hold in large scale; that is, if $f \colon X\to Y$ is a coarsely $n$-to-$1$ map between proper metric spaces $X$ and $Y$ then $\asdim(Y) \leq \asdim(X) + n -1$. Furthermore we  introduce coarsely open coarsely $n$-to-$1$ maps, which include the natural quotient maps via a finite group action, and prove that they preserve the asymptotic dimension.
\end{abstract}

\section{Introduction}

Coarse geometry, particularly asymptotic dimension, has been a powerful tool in the past 20 years in resolving questions of geometric rigidity, and also of geometric group theory. It has been one of the most promising methods of resolving the Baum-Connes Conjecture, which in turn has implications for the Novikov conjecture and Borel conjecture \cite{Yu}. Analogues of the classical covering dimension have provided an essential framework for answering these conjectures. 

An important feature of any dimension are the way in which dimension changes through morphisms. For the classical covering dimension these are described by the Hurewicz dimension theorems. A coarse analogue of the Hurewicz dimension lowering theorem was first proved by Bell and Dranishnikov in \cite{BD} and used to estimate the asymptotic dimension of several classes of groups by applying it to the short exact sequence of groups. The result was later generalized in \cite{BDLM}. On the other hand the coarse versions of the Hurewicz dimension raising theorem and of the finite-to-one mapping theorem depend on a coarse version of $n$-to-$1$ maps, called the coarsely $n$-to-$1$ maps.

A map $f\colon X\to Y$ of metric spaces is said to be coarsely $n$-to-$1$ if it looks like an $n$-to-$1$ map from great distances; formally, if for every $R>0$ there exists $S>0$ such that the preimage of a set of diameter at most $R$ can be covered by at most $n$ sets of diameter at most $S$. 
Coarsely $n$-to-$1$ maps are natural coarse versions of classical $n$-to-$1$ maps. They may be used to classify the asymptotic dimension of a space \cite{MV}. Furthermore, if $G$ is a finite group acting on a metric space $Z$ by coarse maps then the quotient $Z\to G \backslash Z $ is coarse and coarsely $n$-to-$1$, see Example 4.2 in \cite{DV}. Miyata and Virk introduced these maps in \cite{MV} and  established analogues of classical results for $n$-to-$1$ maps in topology. In particular, they proved the coarse version of the Hurewicz finite-to-$1$ mapping theorem and following version of the Hurewicz dimension raising theorem for asymptotic dimension:

\begin{Theorem}[Theorem 1.4 of  \cite{MV}]\label{MVRaising1}
Let $X$ and $Y$ be metric spaces, and let $f\colon X\rightarrow Y$ be coarse, coarsely surjective, and coarsely $n$-to-$1$. Then
$$\asdim (Y) \leq (\asdim (X) + 1)\cdot n -1.$$
\end{Theorem} 

Recall that the Hurewicz dimension lowering theorem in topology has a sharper bound; it says that if $f\colon X\to Y$ is a closed $n$-to-$1$ map of metric spaces then $dim(Y) \leq dim(X) + n-1$. It has been a matter of some debate as to whether the result of Miyata and Virk can be strengthened to obtain the bound of the classical dimension raising theorem. Dydak and  Virk in \cite{DV1} proved such result for hyperbolic groups by considering the induced map on the Gromov boundary. The aforementioned question is the very aim of this paper. In particular, considering the induced map on Higson corona we prove the following result.

\begin{Theorem}
Let $X$ and $Y$ be proper metric spaces and let $f\colon X\rightarrow Y$ be coarse, coarsely surjective, and coarsely $n$-to-$1$. Then
$$\asdim (Y) \leq \asdim (X) +  n -1.$$
\end{Theorem}

Our method of proof is to show that coarsely $n$-to-$1$ maps induce close $n$-to-$1$ maps of Higson Coronas and use a topological version of the Hurewicz dimension lowering theorem, alongside with the  result of Dranishnikov, Keesling, and Uspenskij (\cite{DKU},\cite{Dr}) that the covering dimension of the Higson corona of a proper metric space is equal to its asymptotic dimension. A similar approach was recently independently used by Kasprowski \cite{KA} in order to prove that if a finite group $G$ is acting on a proper metric space $Z$ by isometries then the quotient $Z\to G \backslash Z $ preserves the asymptotic dimension: in this case the induced map on Higson corona is open and one may use the fact that open $n$-to-$1$ maps preserve dimension. We note that the class of coarsely $n$-to-$1$ maps is vastly larger that the class of quotient maps $Z\to G \backslash Z $ by the finite group action (see \cite{MV} or Example \ref{ExRais} for an example of a coarsely $n$-to-$1$ maps strictly raising the asymptotic dimension). Amongst others, it may be used to classify the asymptotic dimension \cite{MV} and contains 'coarse space filling curves'. 
We also introduce coarsely open maps, which correspond to maps for which the induced map on the Higson corona is open, and prove that they preserve the asymptotic dimension in case they are coarsely $n$-to-$1$. They also contain the class of quotient maps $Z\to G \backslash Z $ by the finite group action and as a corollary we deduce the main result of \cite{KA}.
For another application of the Higson corona in coarse geometry see \cite{DW}.

\section{Preliminary Definitions}

We start with an introduction of the coarse category of metric spaces. The reader who wants a more motivated introduction to general large scale structures should look into \cite{ADH}, \cite{DH}, or \cite{R}. All spaces will be considered to be metric unless stated otherwise. 

\begin{Definition}
A collection $\U$ of subsets of a metric space $(X,d)$ is \textbf{uniformly bounded} if there exists $R>0$ so that $\diam(U)\leq R, \forall U \in \U$. Note that $\diam (U)=\sup_{x,y\in U}d(x,y).$ A subset $A\subset X$ is $R$-disjoint for $R>0$ if $d(x,y)>R, \forall x,y\in A.$ A map $f\colon X \to Y$ of metric spaces is \textbf{proper} if for every bounded set $A\subset Y$ its preimage $f^{-1}(A)$ is bounded.
\end{Definition}

\begin{Definition}
The \textbf{coarse category} of metric spaces is defined as follows:
\begin{enumerate}
 \item the objects are metric spaces;
 \item the morphisms are \textbf{coarse} maps: a proper map $f\colon X\to Y$ is said to be \textbf{coarse} (the following equivalent names may be found in published works: \textbf{large scale continuous}; \textbf{ ls-continuous}; \textbf{bornologous})  if $f(\mathcal{U}) := \{f(U):U \in \mathcal{U}\}$ is uniformly bounded for each uniformly bounded cover $\mathcal{U}$ of $X$ (equivalently, if for each $R$ there exists $S$ so that $d(x,y)\leq R$ implies $d(f(x),f(y))\leq S$);
 \item two maps $f,g\colon X \to Y$ are \textbf{coarsely equivalent} (or \textbf{close}) if the collection $\big\{ \{f(x),g(x)\}\mid x\in X \big\}$ is uniformly bounded (equivalently, if there exists $R$ so that $d(f(x),g(x))\leq R, \forall x\in X$);
 \item spaces $X$ and $Y$ are \textbf{coarsely equivalent} if there exist \textbf{coarse equivalences} between them, i.e., maps $f\colon X \to Y$ and $g\colon Y \to X$ so that $f g$ is close to the identity on $Y$ and $g f$ is close to the identity on $X$. 
\end{enumerate}
\end{Definition}

A generic coarse map is not continuous, unless the domain is a discrete space. Since the compactifications naturally arise from subalgebras of continuous functions it will be convenient to work with discrete metric spaces. The following result is well known.

\begin{Proposition}
 Every metric space $X$ is coarsely equivalent to a discrete metric space.
\end{Proposition}

\begin{proof}
 Choose $R>0$ and let $A$ denote a maximal $R-$disjoint set in $X$. Note that $A$ is discrete and coarsely equivalent to $X$ via the following maps: 
\begin{itemize}
 \item   $ X\to A$ maps a point in $X$ to one of the  closest point in $A$;
 \item  $A\to X$  is the natural inclusion.
\end{itemize}
\end{proof}

\begin{Definition}
A map $f\colon X\to Y$  is said to be a \textbf{large scale surjection} (or a \textbf{coarse surjection / coarsely surjective}) if the inclusion mapping $f(X) \hookrightarrow{} Y $ is a coarse equivalence, i.e., if there exists $R$ so that $\inf_{a\in f(X)}d(y,a)\leq R, \forall y\in Y.$ 
\end{Definition}

\begin{Definition}
A metric space $X$  is said to be \textbf{proper} if the closure of each bounded set is compact.
\end{Definition}

The following definition first appeared in \cite{MV}. Coarsely $n$-to-$1$ maps represent the coarse version of $n$-to-$1$ maps.
Note that $n\in \NN$ where $\NN=\{1,2,\ldots\}$.

\begin{Definition}
A map $f\colon X\to Y$ of metric spaces is said to be \textbf{coarsely $n$-to-$1$} if for every uniformly bounded cover $\mathcal{U}$ of $Y$ there exists a uniformly bounded cover $\mathcal{V}$ of $X$ such that $f^{-1}(U)$ can be covered by at most $n$ elements of $\mathcal{V}$ for each $U\in \mathcal{U}$.  An equivalent condition is the following: for each $R$ there exists $S$ so that for each subset $A\subset Y$ of diameter at most $R$, the preimage $f^{-1}(A)$ may be covered by at most $n$ subsets of $X$ of  diameter at most $S$.
\end{Definition}

Note that each coarsely $n$-to-$1$ map is proper. We will be mostly interested in the dimension raising aspects of such maps. For a better understanding we provide an example. 

\begin{Example}\label{ExRais}
 Consider the coarse Cantor set, i.e., the countable direct sum $\oplus_{i\in \NN} \ZZ_2$ of the two-element groups $\ZZ_2$ equipped with any proper metric. We may take, for example, 
 $$
 d\big((a_i)_{i\in\NN},(b_i)_{i\in\NN}\big)=\sum_{i\in \NN} |a_i-b_i|\cdot 2^i.
 $$
 Note that the above sum is always finite. Map $f\colon \oplus_{i\in \NN} \ZZ_2 \to \NN$ defined by $(a_i)_{i\in\NN} \mapsto \sum_{i\in \NN} a_i\cdot 2^i$ is bijective, coarsely $2$-to-$1$, coarse and raises the asymptotic dimension by one. It is a coarse version of the standard $2$-to-$1$ surjection of the Cantor set onto the unit interval, obtained by contracting the intervals which get removed in the inductive construction of the classical Cantor set. For a general example in this spirit see the coarse finite-to-one mapping theorem in \cite{MV}.
\end{Example}

\section{Higson Compactification}

In this section we recall the construction of the Higson compactification and provide some of its properties. 

\begin{Definition}
 A map $f\colon X\to Y$ is \textbf{slowly oscillating} if for every $R,\epsilon >0$ there exists some bounded set $A\subset X$ such that $\diam(f(N(x,R)\setminus A)) \leq \e$ for all $x\in X$.
\end{Definition}

A proper metric space $X$ has a \textbf{Higson compactification}, denoted by $hX$, which is determined by all bounded continuous complex valued slowly oscillating functions on $X$ (the class of such functions will be  denoted by $C_h(X)$).  The Higson compactification $hX$ is the closure of $X$ in $\Pi_{f \in C_h(X)} \overline{im(f)}$ under the embedding $x \to (f(x))_{f\in C_h(X)}$. We denote the Higson corona $\nu X = hX\setminus X$. For an extensive treatment of the theory of compactifications see \cite{E}. 
Let $C(X)$ denote the set of all continuous functions on $X$. Recall that $C(hX) \cong C_h(X)$, i.e., every $f\in C_h(X)$ extends to a continuous function on $hX$ and for every $g\in C(hX)$ the restriction $g|_X$ is contained in $C_h(X)$.  For $A \subset hX$ we define $\nu A=\nu X \cap \overline A$, where $\overline A$ denotes the closure of $A$ in $hX$. Note that if $X$ is discrete then $\nu A$ is homeomorphic to the Higson corona of $A$ as:
\begin{itemize}
 \item $A$ is dense in $A \cup \nu A$;
 \item $A \cup \nu A \subset hX$ is compact Hausdorff;
 \item Any bounded slowly oscillating function $\f$ on $A$ extends over $A \cup \nu A$ to a bounded continuous function, hence the universal property holds. In order to prove it, extend $\f$ over $X$ by  \cite[Corollay 4.9]{DM} to a slowly oscillating function, which in turn extends over $hX$ to a continuous function. Restrict the resulting function to $A \cup \nu A$ in order to obtain the required extension of the original function $\f$.
\end{itemize}

\begin{Remark} 
 Mine and Yamashita in \cite{MY} prove a remarkable result concerning the Higson corona functor. Let $\mathcal{A}$ be the category whose objects are $X_0$ where $X$ is a totally bounded metric space (here $X_0$ denotes $X$ with the $C_0$ coarse structure) and whose morphisms are coarse maps. They show that the Higson corona functor is an equivalence of $\mathcal{A}$ and the category of compact metric spaces with continuous maps. 
 It follows that the Higson corona functor preserves all coarse information as topological information for a large class of coarse spaces.
\end{Remark}

If $X$ is a metric space, $x\in X$, and $A\subset X$ we define  $N(A,R) = \{y \in X:d(y,A) \leq R\}$ and $N(x,R)=N(\{x\},R)$.

The first part of the following proposition is well known. For an alternative proof see \cite{R}.

\begin{Proposition}
 A proper coarse map $f\colon X \to Y$ between proper discrete metric spaces extends to a continuous map $hf \colon hX \to hY$. The extended portion is a continuous map between coronas $\nu f = hf |_{\nu X} \colon \nu X \to \nu Y$.
 
 Additionally, $f$ is coarsely surjective if and only if $\nu f$ is surjective. 
\end{Proposition}

\begin{proof}
 Choose $\f\in C_h(Y)$. We claim that $f \f \in C_h(X)$. Choose $R,\e>0$. Let $S>0$ with the property that $\diam(f(A))\leq S$ for each $A\subset X, \diam(A)\leq R$. Choose a bounded set  $C\subset Y$ corresponding to the function $\f$ and parameters $S,\e$, i.e., $\diam(\f(B(y,S)\setminus C)) \leq \e$ for all $y\in Y$. Note that $\diam(f \f(B(x,R)\setminus f^{-1}(C))) \leq \e$ for all $x\in X$, and $f^{-1}(C)$ is bounded as $f$ is proper. By the standard theory of compactifications (see \cite{E}) the required map $hf$ exists.
 
 We next prove that $\nu f$ is surjective provided $f$ is coarsely surjective with a constant $M$, i.e., $ \forall y  \in Y \quad \exists x\in X: d(y,f(x))<M$. Fix $\tilde y\in \nu Y$. Choose an open set 
 $$
 V_\e=\bigcap_{i=1}^k \f_i^{-1}\big(N(\f_i(\tilde y), \e)\big) \subset hY
 $$
 for some $\f_i \in C(hX), \e>0$. Note that sets $V_\e$ form a basis of neighborhoods of $\tilde y \in hY$. For each $i$ choose a bounded $C_i \subset Y$ so that $\diam\big(\f_i(B(z, 2M) \setminus C_i) \big) \leq \e, \forall z\in Y$. Define 
 $$
 C=\bigcup _{i=1}^k N(C_i,2M).
 $$
 We may choose $\hat y \in V_\e \setminus C$ as $C$ is bounded and $V_\e$ is not. We may also choose $\hat x\in X$ so that $d(f(\hat x),\hat y)\leq M$. Therefore
 $$
 d\big(  \f_i(f(\hat x)), \f_i(\tilde y) \big) \leq 
  d\big(  \f_i(f(\hat x)), \f_i(\hat y)  \big) +  
  d\big(  \f_i(\hat y), \f_i(\tilde y)  \big) \leq
  \e+\e = 2 \e,
 $$
 hence $f(\hat x)\in V_{2 \e}$. Since the choice of $\e$ and of the finite collection of  functions $\f_i$ was arbitrary, and since $hf(hX)$ is compact we have $\tilde y \in hf(hX)$, which means that $\nu f$ is surjective.
 
Suppose $\nu f$ is not surjective. Fix $z_0 \in Y$ as a basepoint and define $B_r = N(z_0,r)$. Take any $R>0$. Choose $y\in \nu Y \setminus \nu f(\nu X)$. Since $\nu f(\nu X)$ is closed we may find $\f\in C(hY)$ so that $\f(y)=2$ and $\f (\nu f(\nu X))=0$. Since $y$ is in the closure of $Y$ in $hY$ we may find $q_0\in \NN$ so that for each $q\in \NN, q>q_0$ there exists some $y_q\in Y\setminus B_q$ so that $\f(y_q)>1$. Since $\f|_Y$ is slowly oscillating we may choose $q_1\in \NN$ so that for all $x,y\in Y\setminus B_{q_1}, d(x,y)< R$ we have $|\f(x)-\f(y)| \leq 1/2.$ If for all $y_q, q \geq \max \{q_0, q_1\}$ there was $x_q\in X$ so that $d(f(x_q),y_q) \leq R$ then $\f (f(x_q))\geq 1/2$ and since $x_q$ would have been unbounded by the coarseness of $f$, the set $\nu \{f(x_q)\}_{q \geq \max \{q_0, q_1\}}$ would be nonempty and mapped to values greater or equal to $1/2$. That is a contradiction with the choice of $\f$ hence such $R$ may not exist. Since the choice of $x_q$ was arbitrary we conclude that $f$ is not coarsely surjective.
 \end{proof}

\section{The topology of Higson compactification}

In this section we develop a metric condition (see Proposition \ref{graduallydisjoint}) equivalent to the fact that  disjoint sets in $X$ have disjoint closures in $hX$. The condition is called gradual disjointness. We will also use a metric condition (see Proposition \ref{gradint}) equivalent to the fact that the intersection of coronas of sets in $X$ have nonempty intersection. In this case the sets are said to diverge (see \cite{DKU}). Subsection 4.1 generalizes these notions to countable families. Subsection 4.2 introduces coarsely open maps, which correspond to maps $f$ for which $\nu f$ is open.

\begin{Definition}
Suppose $\A = \{A_1, A_2, \ldots, A_k\}$ is a collection of disjoint subsets of a metric space $X$. Collection $\A$ is
 \textbf{gradually disjoint}  if for each $R>0$ there exists a bounded subset $B\subset X$ so that the sets $N(A_i,R) \setminus B$ are disjoint.
\end{Definition}

It is easy to see that if a collection contains a bounded set then it is gradually disjoint.

\begin{Lemma}\label{lemma55}
 A collection $\A = \{A_1, A_2, \ldots, A_k\}$  is gradually disjoint if and only if for each $i\neq j$ the collection $\{A_i, A_j\}$ is gradually disjoint.
\end{Lemma}

\begin{proof}
 The proof is easy and left to the reader.
\end{proof}

\begin{Definition}\cite{DKU}\label{DefGradInt}
Suppose $\A = \{A_1, A_2, \ldots, A_k\}$ is a collection of  subsets of a metric space $X$. Collection $\A$ 
 \textbf{diverges}  if for each $R>0$ there exists a bounded subset $B\subset X$ so that 
 $$
 \bigcap_{i=1}^k N(A_i,R) \setminus B = \emptyset, 
 $$
 i.e., if for each $R>0$ the intersection $ \bigcap_{i=1}^k N(A_i,R)$ is bounded.
\end{Definition}

\begin{Proposition}\cite[Proposition 2.3]{DKU}\label{gradint}
Suppose $\A = \{A_1, A_2, \ldots, A_k\}$ is a collection of  subsets of a proper discrete metric space $X$. Then $\A$ diverges if and only if $\bigcap_{i=1}^k \nu A_i = \emptyset$.
\end{Proposition}

\begin{Proposition}\label{graduallydisjoint}
Suppose $\A = \{A_1, A_2, \ldots, A_k\}$ is a collection of disjoint subsets of a proper discrete metric space $X$. Then $\A$ is gradually disjoint if and only if the sets $\nu A_i$ are disjoint.
\end{Proposition}

\begin{proof}
By Lemma \ref{lemma55} it suffices to consider collections of pairs.  By Definition \ref{DefGradInt} a collection $\{A_m, A_j\}$  is gradually disjoint if and only if it diverges, which is equivalent to $\nu A_m \cap \nu A_j = \emptyset$ by  Proposition \ref{gradint}.
\end{proof}



\subsection{Countable families}
In this subsection we generalize divergent and gradually disjoint families. We also provide an alternative proof of Proposition \ref{gradint} which follows from a  generalization,  proved using the theory developed in the previous subsection. 

\begin{Definition}
Suppose $\A = \{A_1, A_2, \ldots\}$ is a collection of disjoint subsets of a metric space $X$. Collection $\A$ is
 \textbf{gradually disjoint}  if for each $R>0$ and for each finite $\B \subset \A$ there exists a bounded subset $B\subset X$ so that the sets $\{N(A_i,R) \setminus B\}_{A_i\in\B}$ are disjoint, i.e., if each finite subfamily $\B$ is gradually disjoint, or equivalently, if each pair is gradually disjoint.
\end{Definition}

\begin{Proposition}\label{graduallydisjointCount}
Suppose $\A = \{A_1, A_2, \ldots\}$ is a collection of disjoint subsets of a proper discrete metric space $X$. Then $\A$ is gradually disjoint if and only if the sets $\nu A_i$ are disjoint.
\end{Proposition}

\begin{proof}
 Choose a finite $\B \subset \A$. By Proposition \ref{graduallydisjoint} a collection $\B$ is gradually disjoint if and  only if the sets $\{\nu A_i\}_{A_i\in\B}$ are disjoint. Consequently $\A$ is disjoint.
\end{proof}

%
%
%


\begin{Definition}\label{DefGradIntCount}
Suppose $\A = \{A_1, A_2, \ldots\}$ is a collection of  subsets of a metric space $X$. Collection $\A$ 
 \textbf{diverges}  if there exists a finite $\B \subset \A$ so that for each $R>0$ there exists a  bounded subset $B\subset X$ so that
  $$
 \bigcap_{A_i\in\B} N(A_i,R) \setminus B = \emptyset.
 $$
 Equivalently, a collection $\A$ diverges if there exists a finite $\B \subset \A$ that diverges.
 \end{Definition}


\begin{Lemma}\label{DivCount}
Suppose $X$ is a proper discrete metric space. A collection $\{A_i\}_{i\in \NN}$ of disjoint subsets of $X$ diverges if and only if $\bigcap_{i\in \NN} \nu A_i = \emptyset.$
\end{Lemma}

\begin{proof}
Since sets $\nu A_i$ are closed subsets of a compact space we may use the intersection property: $\bigcap_{i\in \NN} \nu A_i = \emptyset$ if and only if there exists a finite $\B \subset \A$ for which $\bigcap_{A_i\in\B} \nu A_i = \emptyset.$ By Proposition \ref{gradint} this is equivalent to the fact that $\A$  diverges.
\end{proof}

 \subsection{Coarsely open maps}

 Throughout this subsection let $x_0\in X$ be a fixed basepoint.
 
\begin{Definition}
A function $\rho\colon [0,\infty) \to [0,\infty)$ is \textbf{TI-increasing} (towards infinity increasing) if it is increasing (not necessarily strictly increasing) and $\lim_{t\to\infty}\rho(t)=\infty$.

 Given $A\subset X$ and a TI-increasing $\rho \colon [0,\infty) \to [0,\infty)$ define 
 $$
 N(A,\rho)=\bigcup _{x\in A} N\Big(x, \rho (d(x,x_0))\Big).
 $$
\end{Definition}

\begin{Remark}\label{basis}
Let $X$ be a proper metric space (or, more generally, a proper coarse space). The collection
 $$
 \V = \{  \f^{-1}([0,\eps)) \cap \nu X \mid \eps > 0, \f\colon hX \to [0,1]  \textit{ continuous}\}.
 $$
 is a basis for $\nu X$ by the following two arguments:
\begin{itemize}
 \item It is, by definition, a subbase for the topology on $\nu X$. 
 \item It is closed under intersections: if $f,g\colon hX\to [0,1]$ and $1 > \eps > 0$, then 
 $$
 (f\cdot g)^{-1}([0,\eps)) \subset f^{-1}([0,\eps))\cap g^{-1}([0,\eps)).$$
\end{itemize}
 \end{Remark} 
 

\begin{Proposition}\label{PropBase}
Suppose $X$ is a proper discrete metric space. The collection 
 $$
 \U=\{\nu N(A,\rho) \mid A \subset X, \rho \colon (0,\infty) \to (0,\infty) \textit{ TI-increasing}\}
 $$ 
 is a base for the topology on $\nu X.$
\end{Proposition}

\begin{proof}
By remark \ref{basis}, the topology of $\nu X$ is given by 
 $$
 \V = \{  \f^{-1}([0,\eps)) \cap \nu X \mid \eps > 0, \f\colon hX \to [0,1]  \textit{ continuous}\}.
 $$
 It thus suffices to compare $\U$ and $\V$.
 
 Choose $x\in X, \eps>0$ and a continuous $\f\colon hX \to [0,1], \f(x)=0$. Define $A=\f^{-1}([0,\eps/3))$. As $\f|_X$ is slowly oscillating there exists for each $R >0$ some $S_R>0$ so that $\diam (\f(N(x',R)\setminus N(x_0,S_R))) \leq \eps/3, \forall x'\in X$. Define $\r(S_R+R)=R$ (without loss of generality we may assume that $\rho$ defined in such way is increasing towards infinity) and note that $N(A,\rho)\subset \f^{-1}([0,2\eps/3))$ hence $x\in \nu N(A,\rho)\subset  \f^{-1}([0,\eps))$.
 
 Choose $A\subset X, x\in \nu A, \rho \colon (0,\infty) \to (0,\infty)$ increasing, $\lim_{t\to\infty} \rho(t)=\infty$. We will construct a slowly oscillating $\f$ so that $x\in \f^{-1}([0,1/2)) \cap \nu X \subset \nu N(A,\rho)$. Define $\f(N(A, 1/3 \cdot \rho))=0$ and $\f(X \setminus N(A,\rho))=1$. By the properties of $\rho$ the sets $N(A, 1/3 \cdot \rho)$ and  $X \setminus N(A,\rho)$ are gradually disjoint hence 
 $$
 \f\colon \Big(N(A, 1/3 \cdot \rho) \cup X \setminus N(A,\rho)\Big)\to[0,1]
 $$
 is slowly oscillating. Extend it over $X$  by  \cite[Corollary 4.9]{DM} and over $hX$ by the universal property. Let $\f$ also denote the obtained extension. Note that $\f^{-1}([0,1/2))\cap  X\subset N(A,\rho)$ hence $x\in \f^{-1}([0,1/2)) \cap \nu X \subset \nu N(A,\rho)$.
\end{proof}

\begin{Proposition}\label{PropBase1}
Suppose $X$ is a proper discrete metric space and $A\subset X$ is unbounded. The collection 
 $$
 \U=\{\nu N(A,\rho) \mid \rho \colon (0,\infty) \to (0,\infty) \textit{ TI-increasing}\}
 $$ 
 is a base for the collection of neighborhoods of $\nu A \subset \nu X.$
\end{Proposition}

\begin{proof}
 The proof is almost identical to the proof of Proposition \ref{PropBase}.
\end{proof}

\begin{Definition}\label{CoarselyOpen}
  Suppose $f\colon X\to Y$ is a map between a proper  metric spaces. Map $f$ is \textbf{coarsely open} if for each unbounded $A\subset X$ and for each TI-increasing $\rho \colon [0,\infty) \to [0,\infty)$ there exists a TI-increasing $\tilde\rho \colon [0,\infty) \to [0,\infty)$ so that $N(f(A), \tilde \rho) \subset f(N(A,\rho))$.
\end{Definition}

\begin{Proposition}\label{InducedOpen}
Suppose $f\colon X\to Y$ is a map between proper discrete metric spaces. If $f$ is coarsely open then $\nu f$ is open.
\end{Proposition}

\begin{proof}
We will use Propositions \ref{PropBase} and \ref{PropBase1} throughout the proof.

Suppose $f$ is coarsely open. Choose $x\in \nu X$ and an open subset $U \subset \nu X$ containing $x$. We may assume that $U=\nu N(A,\rho)$ for some choice of $A$ with $x\in \nu A$ and $\rho$. The existence of $\tilde \rho$ as in Definition \ref{CoarselyOpen} means that $f(x)$ is contained in an open set $N(f(A), \tilde \rho) \subset f(N(A,\rho))$. Hence $f(U)$ is a neighborhood of $f(x)$. Since $x$ and $U$ were aribtrary we conclude that $\nu f$ is open.
%
\end{proof}
 
%
%

\begin{Proposition}\label{OpenGroupAction}
 If a finite group $G$ acts on $X$ by coarse maps then $X \to G\setminus X$ is a coarsely open map.
\end{Proposition}

\begin{Remark}
We use notation $G\setminus X$ to denote the space of orbits $\{G\cdot x\}_{x\in X}$ equipped with the Hausdorff metric. The notation is used in \cite{KA} and   distinguishes the space of orbits from the quotient $X/G$ (which denotes the 'right' orbits) in case when $X$ is a group and $G$ is its subgroup. However, in \cite{DV} the notation $X/G$ is used (instead of $G\setminus X$) to denote the space of orbits as there is no reference to the subgroup case. 
\end{Remark}

\begin{proof}
As in Example 4.2 of \cite{DV} we may assume that $G$ acts on $X$ by isometries. According to Proposition 3.1 of \cite{KA} $G$ induces an action on $\nu X$ and the map induced by $X \to G \setminus X$ is the quotient map $\nu X \to G \setminus \nu X$, which is open. By Proposition \ref{InducedOpen}  map $X \to G \setminus X$ is coarsely open.
\end{proof}

It is easy to construct examples of coarsely open maps which are not quotients by a finite group action. Consider $\CC$ as a domain and identify $z$ with $\bar z$ for all $z$, whose real component is non-positive. The resulting quotient map is coarsely open, if we equip the quotient space with the Hausdorff metric.  

\section{Main Results}
The following is the main technical result of the paper. 

\begin{Theorem}\label{compactification1}
Let $X$ and $Y$ be proper discrete metric spaces. Suppose $f\colon X\to Y$ is coarse. 
Then the following are equivalent:
\begin{enumerate}
 \item $f$ is coarsely $n$-to-$1$;
 \item $\nu f$ is  $n$-to-$1$;
 \item If $\A=\{A_1, A_2, \ldots, A_k\}$ is a gradual disjoint collection of subsets of $X$ for which $f(\A)=\{f(A_1), f(A_2), \ldots, f(A_k)\}$ does not diverge, then $k \leq n$.
\end{enumerate}
\end{Theorem}

\begin{proof}
 Fix basepoints $\bar x\in X, \bar y\in Y$. For $r>0$ define $C_r=N(\bar y,r)$. Given any $A\subset X$ note that $\nu f(\nu A) = \nu f(A)$, where the left hand side of the  equality represents map $\nu f$ evaluated at set $\nu A$, while the right hand side represents the boundary in the Higson corona, usually denoted by $\nu$, of the set $f(A)$. We will use this equality throughout the proof.

 {(3) $\implies$ (2):}
Assume there exist $y\in Y$ and  distinct points $x_0, x_1, \ldots, x_n\in \nu X$ with $\nu f(x_i)=y, \forall i$. Choose disjoint closed neighborhoods $A_i$ if $x_i$ in $hX$. Note that the collection $\A=\{A_i\cap X\}_{i=0,1,\ldots,n}$ is gradually disjoint by Proposition \ref{graduallydisjoint}  since 
$$
\bigcap_{i=0}^n \nu(A_i\cap X) \subset \bigcap_{i=0}^n A_i = \emptyset
$$
On the other hand 
$$
y \in \bigcap_{i=0}^n \nu f (\nu A_i) = \bigcap_{i=0}^n \nu f(A_i) $$
therefore sets $f(A_i)$ do not diverge by Proposition \ref{gradint} as $\nu f (\nu A_i)= \nu f(A_i), \forall i$, thus completing a proof by contradiction. 

 {(2) $\implies$ (3):}
Assume there exists a gradually disjoint collection $\{A_0, A_1, \ldots, A_n\}$ of subsets of $X$ so that the collection $\{f(A_0), f(A_1), \ldots, f(A_n)\}$ does not diverge. By Propositions \ref{graduallydisjoint} and \ref{gradint}  sets $\nu A_i$ are disjoint while $\nu f(\nu A_i)$ have a common intersection, thus $\f$ is not $n$-to-$1$.

 {(1) $\implies$ (3):}
  Suppose $\{A_0,A_1,\ldots, A_n\}$ is a gradually disjoint family of subsets of $X$ for which $\{f(A_0),f(A_1),\ldots, f(A_n)\}$ does not diverge. Then there exists $S>0$ such that for every $R \in \NN$ there is $x_{R,i} \in A_i\setminus C_R$ such that $diam\{f(x_{R, i}): 0\leq i\leq n\} \leq S$. Using coarseness we deduce that set $\{f(x_{R, i}): R\in \NN\}$ is unbounded for each $0\leq i\leq n$. By the gradual disjointness of the sets $A_i$ for $0\leq i\leq n$, we have that $\lim\limits_{R\to \infty}d(x_{R,i},x_{R,j}) = \infty$ for $i \neq j$. It follows that $f$ cannot satisfy the coarsely $n$-to-$1$ property for the uniformly bounded collection $\{N(f(x_{R, 1}),2S):R \in \NN\}$.

 {(3) $\implies$ (1):} Suppose $f$ is not coarsely $n$-to-$1$. Then there exists $R>0$ so that for each $S\in \NN$ there exist $x_{S,0}, x_{S,1}, \ldots, x_{S,n}$ so that:
\begin{description}
 \item [(a)]$d(x_{S,i},x_{S,j})> S,\quad \forall i \neq j$;
 \item [(b)] $d(f(x_{S,i}),f(x_{S,j}))< R,\quad \forall i, j$.
\end{description}
For $S\in \NN$ define $A^S=\{x_{S,i}\}_{i=0,1,\dots,n}, B^S=\{f(x_{S,i})\}_{i=0,1,\dots,n}$ and note that $\diam (A^S) > S, \diam (B^S) < nR$. Since $f$ is proper the collection $B^S$ is unbounded, i.e., $d(\bar y, B^S)\to \infty$  as $S \to \infty$. Therefore $d(\bar x_0, x_{r,k})\to \infty$ as $r\to\infty$ for each $k$ as $f$ is coarse. We may thus inductively pass to  a subsequence of $\NN$ so that the following condition holds:
$$
 \hspace{1.5cm}d(x_{r,i},x_{q,j})>r, \quad \forall q \leq r, \forall i,j, \textrm{ except for } (r,i)=(q,j). \hspace{2cm} (*)
$$
Define $ A_k=\{x_{r,k}\}_{r\in \NN}$ and note that $\{A_0,A_1,\ldots, A_n\}$ is gradually disjoint by  $(*)$. However, $\{f(A_0),f(A_1),\ldots, f(A_n)\}$ does not diverge due to (b).
\end{proof}

\subsection{Dimension raising theorem for coarsely $n$-to-$1$ maps}

An important application of Theorem \ref{compactification1} is the optimized  version of the coarse dimension raising theorem. Recall the classical Hurewicz dimension raising theorem for compact Hausdorff spaces.

\begin{Theorem}[\cite{E}  Theorem 3.3.7 on page 196] \label{Hurewicz-raising}
Let $f\colon X\rightarrow Y$ be a closed surjective map between normal spaces such that $|f^{-1}(y)| \leq n+1$ for each $y\in Y$.
Then $\dim Y \leq \dim X + n $.
\end{Theorem}

 \cite{MV} introduced the coarse version of the dimension raising theorem.

\begin{Theorem}[Theorem 1.4 of  \cite{MV}]\label{MVRaising}
Let $X$ and $Y$ be metric spaces. Suppose  $f\colon X\rightarrow Y$ is coarse, coarsely $n$-to-$1$  and coarse surjective. Then
$$\asdim (Y) \leq (\asdim (X) + 1)\cdot n -1.$$
\end{Theorem}

Theorem \ref{MVRaising} provided an upper bound for $\asdim (Y)$. However, the bound differed from the classical version and over the past few years there has been an ongoing debate on the possibility of its improvement. A coarse version of the classical proof seems unfeasible. The classical bound was achieved for the special case in \cite{DV1} via the Gromov boundary $\partial X$ of a proper hyperbolic geodesic space.

\begin{Theorem}[\cite{DV1}] \label{DimRaiDV}
Suppose $X$, $Y$ are proper $\delta$-hyperbolic geodesic spaces and $f\colon X \to Y$ is a radial function. If $f$ is coarsely $(n+1)$-to-1 and coarsely surjective then
$$
\dim(\partial Y)\leq \dim (\partial X)+n.
$$

Moreover, if $X$ and $Y$ are hyperbolic groups then
$$
\asdim(Y)\leq \asdim (X)+n.
$$
\end{Theorem}

One of the issues connected to  Theorem \ref{DimRaiDV} was the fact that the relation between $\dim(\partial X)$ and $ \asdim (X)$ is well established only for hyperbolic groups. In the case of Higson compactification the following result allows us to deduce a much more general coarse version of the dimension raising theorem with the classical bound.  

\begin{Theorem}[\cite{Dr,DKU}]\label{HigsonDim}
 If $X$ is a proper metric space with  $\asdim (X) <\infty$ and $n \geq 0$, then the
following conditions are equivalent:
\begin{itemize}
 \item $\dim(\nu X) \leq n;$
 \item $\asdim (X) \leq n.$
\end{itemize}
\end{Theorem}

Combining all the mentioned results we obtain the following dimension raising theorem.

\begin{Theorem}\label{Raise}
Let $X$ and $Y$ be proper metric spaces and let $f\colon X\rightarrow Y$ be coarse, coarsely $n$-to-$1$ and coarsely surjective. Then
$$\asdim (Y) \leq \asdim (X) +  n -1.$$
\end{Theorem}

\begin{proof}
Without loss of generality we may assume $X$ and $Y$ to be discrete: if they are not we may replace them by coarsely equivalent discrete proper metric spaces spaces. The map  between such spaces, which is naturally induced  by $f$ is still coarse, coarsely $n$-to-$1$ and coarsely surjective.

If $\asdim (X)$ is not finite then the statement is trivial. Suppose $\asdim(X)$ is finite, which implies that $\asdim(Y)$ is finite as well by Theorem \ref{MVRaising}. Map $f$ induces a continuous $n$-to-$1$ map between Higson coronas by Theorem \ref{compactification1}. Since  Higson coronas are compact and Hausdorff  the induced map is closed. Using  Theorem \ref{Hurewicz-raising}  we obtain $\dim (\nu Y) \leq \dim (\nu X) +  n -1.$
Using  Theorem  \ref{HigsonDim} we obtain $\asdim (Y) \leq \asdim (X) +  n -1.$
\end{proof}

\subsection{Dimension preserving theorem for coarsely open coarsely $n$-to-$1$ maps}

In case when the induced map on the Higson corona is open we may use the same argument to obtain the equality of asymptotic dimension using the following result. 

\begin{Proposition} \label{DimEq}
  \cite[Proposition 9.2.16]{P} Let $X,Y$ be weakly paracompact, normal spaces. Let $f \colon X \to Y$ be a continuous, open surjection. If for every point $y \in Y$ the preimage $f^{-1}(y)$ is finite, then $\dim(X) = \dim(Y)$.
\end{Proposition}

\begin{Theorem}\label{RaiseOpen}
Let $X$ and $Y$ be proper metric spaces and let $f\colon X\rightarrow Y$ be coarse, coarsely open, coarsely $n$-to-$1$ and coarsely surjective. Then
$$\asdim (Y) = \asdim (X).$$
\end{Theorem}

\begin{proof}
Without loss of generality we may assume $X$ and $Y$ to be discrete: if they are not we may replace them by coarsely equivalent discrete proper metric spaces spaces. The map  between such spaces, which is naturally induced  by $f$ is still coarse, coarsely open, coarsely $n$-to-$1$ and coarsely surjective.

If $\asdim (X)$ is not finite then the statement is trivial. Suppose $\asdim(X)$ is finite, which implies that $\asdim(Y)$ is finite as well by Theorem \ref{MVRaising}. Map $f$ induces a continuous open $n$-to-$1$ map between Higson coronas by Theorem \ref{compactification1} and Proposition \ref{InducedOpen}.  Using  Theorem \ref{DimEq}  we obtain $\dim (\nu Y) = \dim (\nu X).$
Using  Theorem  \ref{HigsonDim} we obtain $\asdim (Y) = \asdim (X).$
\end{proof}

As a special case we obtain the main result of \cite{KA}.

\begin{Corollary}\cite[Theorem 1.1]{KA}
Let $X$ be a proper metric space and let $G$ be a finite group acting by isometries  on $X$. Then $G\setminus X$ has the same asymptotic dimension as $X$.
\end{Corollary}

\begin{proof}
The natural map $X \mapsto G\setminus X$ is:
\begin{itemize}
 \item coarsely surjective as it is surjective, 
 \item coarse by the definition of the Hausdorff metric on $G \setminus X$,
 \item coarsely $|G|$-to-$1$ by \cite{DV},
 \item coarsely open by Proposition \ref{OpenGroupAction}.
\end{itemize}
The conclusion follows by Theorem \ref{RaiseOpen}.
\end{proof}


\subsection{Dimension preserving theorem in case the induced map on the Higson corona is finite-to-$1$}\label{Open}

\begin{Theorem}\label{compactification2}
Let $X$ and $Y$ be proper discrete metric spaces. Suppose $f\colon X\to Y$ is coarse. 
Then the following are equivalent:
\begin{enumerate}
 \item $\nu f$ is  finite-to-$1$;
 \item there exist no gradually disjoint collection $\A=\{A_1, A_2, \ldots\}$ of subsets of $X$ for which $f(\A)=\{f(A_1), f(A_2), \ldots, \}$ does not diverge.
\end{enumerate}
\end{Theorem}

\begin{proof}
$(1)\Rightarrow (2)$ Suppose that there exists a gradually disjoint collection $\A=\{A_1, A_2, \ldots\}$ of subsets of $X$ for which $f(\A)=\{f(A_1), f(A_2), \ldots, \}$ does not diverge. By Propositions \ref{graduallydisjointCount} and  \ref{DivCount} the induced map $\nu f$ could not be finite-to-$1$.

$(2)\Rightarrow (1)$  Suppose that there exists $y\in \nu Y$ which has countably many preimages, say $\nu f^{-1}(y) \supseteq Q=\{x_1,x_2,x_3, \hdots \}$. Without  loss of generality (by potentially taking an appropriate infinite subset) we may assume that $Q\subset \nu X$ is a discrete subspace. Let $\{V_1, V_2, \ldots \}$ denote a collection of disjoint closed neighborhoods of points $x_1, x_2, \ldots $ respectively, taken in $h(X)$. By Propositions \ref{graduallydisjointCount} and  \ref{DivCount}  the family $\{V_i \cap X\}_{i\ge 1}$ is gradually disjoint while $\{f(V_i \cap X)\}_{i\ge 1}$ is not divergent. 
\end{proof}

\begin{Theorem}\label{RaiseCountable}
Let $X$ and $Y$ be proper metric spaces of finite asymptotic dimension and let $f\colon X\rightarrow Y$ be coarse, coarsely open, coarsely surjective, and satisfies the property that there exist no divergent disjoint collection $\A=\{A_1, A_2, \ldots\}$ of subsets of $X$ for which $f(\A)=\{f(A_1), f(A_2), \ldots, \}$ does not diverge  then
$$\asdim (Y) = \asdim (X).$$
\end{Theorem}
\begin{proof}
By Theorem \ref{compactification2}, we know that $\nu f$ is a finite-to-$1$ map. By Theorem \ref{DimEq}, we have $dim(\nu X) = dim(\nu Y)$. Because $asdim(X),asdim(Y) < \infty$, Theorem \ref{HigsonDim} gives that $asdim(X) = dim(\nu X) = dim(\nu Y) = asdim(Y)$.
\end{proof}


There exists a concept of coarsely finite-to-$1$ maps. These were defined in \cite{AV} and have been shown to have many dual properties to finite-to-$1$ maps in topology. The following is the definition of coarsely finite-to-$1$ maps as defined in that paper.

\begin{Definition}
A map $f\colon X\to Y$ of metric spaces is \textbf{coarsely finite-to-$1$} if for every $R>0$, there exists $S>0$ and $m$ such that the preimage of each subset of $Y$ with diameter $R$ can be covered by at most $m$ subsets of $X$ with diameter at most $S$. 
\end{Definition}

The following example demonstrates that Theorem \ref{RaiseOpen} cannot be generalized to coarsely finite-to-$1$ maps as defined above as the induced map on the Higson corona is not finite-to-$1$ and it also raises the dimension.

\begin{Example}
Define 
$$
A_k= [k,\infty) \times \{k\}, \qquad B_k= \{k\}\times [k-1,k].
$$ 
Note that $X=\bigcup _{k=1}^\infty (A_k \cup B_{k+1})$ is a connected infinite tree. Let $d$ denote the path metric on it and let $d_e$ denote the induced Euclidean metric. Note that $\asdim (X,d)=1$ (since $(X,d)$ is a length tree) and $\asdim(X,d_e)=2$ as it is coarsely equivalent to one eighth of the plane $\{(x,y)\in \RR^2 \mid x,y\geq 0, x\geq y\}$. Consider the identity map $f\colon (X,d)\to (X,d_e)$. It has the following properties:
\begin{itemize}
 \item $f$ is coarse as it is contraction;
 \item $f$ is coarsely finite-to-$1$ by the following argument: choose $D\in \NN$ and let $W\subset (X,d_e)$ be a set of diameter at most $D$. Then $f^{-1}(W)$ intersects at most $(D+1)$-many intervals $A_k \cup B_{k+1}$ (which form a partition of $X$) and the intersection with each of these sets is contained in some interval of length at most $D+1$. 
 \item $\nu f$ is not finite-to-$1$ by Theorem \ref{compactification2} applied to collection $\{A_1, A_2, \ldots\}$.
\end{itemize}

However, $f$ is not coarsely open which suggests the following question.
\end{Example}

\begin{Question}
Suppose $X$ and $Y$ are proper metric spaces of finite asymptotic dimension and $f\colon X\rightarrow Y$ is coarse, coarsely open, coarsely surjective, and coarsely finite-to-$1$. Can we conclude that $asdim(X) = asdim(Y)$?
\end{Question}

\section*{Acknowledgments} The authors would like to thank the anonymous referee. Their comments/corrections greatly improved the quality of this paper. 

The authors would also like to sincerely thank Nicol\` o Zava and Damian Sawicki, each of whom independently discovered an error in the published version and gracefully informed us about it. The converse of Proposition \ref{InducedOpen} stated in this corrected version does not hold. The mentioned proposition has been stated as equivalence in the published version and accompanied by Corollary 4.16. In this version only the valid direction of the published Proposition \ref{InducedOpen} has been presented while the published Corollary 4.16 has been removed. As the removed statements only attempted to conveyed a side result, the rest of the paper (including the main argument) is not affected. 


\end{document}